\documentclass[11pt]{amsart}
\usepackage{epsfig,color}
\usepackage{blindtext}
\usepackage{graphicx}
\usepackage{enumitem}
\usepackage{url}
\usepackage{amssymb}
\usepackage{graphicx,import}
\usepackage{comment}
\usepackage{esint}
\usepackage{xypic}

\usepackage[margin = 1.4in] {geometry}

\usepackage{hyperref}

\newtheorem{thm}{Theorem}[section]

\newtheorem{lem}[thm]{Lemma}

\newtheorem{rem}[thm]{\bf{Remark}}

\numberwithin{equation}{section}

\theoremstyle{theorem}

\newtheorem{theorem}{Theorem}[section]

\newtheorem{proposition}[theorem]{Proposition}

\newtheorem{corollary}[theorem]{Corollary}

\theoremstyle{definition}

\newtheorem*{remark*}{Remark}

\numberwithin{equation}{section}

\author{Elham Matinpour}
\address{Department of Mathematics, Johns Hopkins University, 3400 N. Charles Street, Baltimore, MD 21218}
\email{ematinp1@jhu.edu}

\title{On the Area of immersed Minimal annuli in a slab}

\begin{document}

\maketitle

\begin{abstract}
In this note, we organize minimal annuli in a slab based on the winding number of the circles that foliate them and study the area of minimal annuli with given winding number. Specifically, we deduce some results regarding the convexity of the length function of corresponding level curves, and apply them to estimate the area of the annuli by comparing to the area of waist of covers of catenoids.
\end{abstract}

{\footnotesize
}

\section{Introduction}
The first nontrivial minimal surface discovered is an annulus in $\mathbb{R}^3$, the catenoid. A result of the work of Schoen \cite{schoen1983uniqueness} features the catenoid as the unique minimal surface with finite total curvature and two embedded ends, and Lopez and Ros \cite{lopez1991embedded} show that the catenoid is the unique complete planar domain of finite total curvature embedded in $\mathbb{R}^3$. There are more examples of minimal annuli embedded in a slab such as a neck of a Riemann example. Rotationally symmetric compact pieces of the catenoid are characterized in a slab in several ways. For instance, Pyo showed that the catenoid in a slab is the only minimal annulus meeting the boundary of the slab in a constant angle \cite{pyo2009minimal}. Also it was proved by Bernstein and Breiner that all embedded minimal annuli in a slab have area bigger than or equal to the minimum area of the catenoids in the same slab \cite{bernstein2014variational}. That minimum is attained by the catenoidal waist along the boundary of which the rays from the center of the slab are tangent to the waist. This waist is marginally stable because it is stable but any subset of the catenoid properly containing the waist is unstable, we call them the maximally symmetric marginally stable piece of catenoid in that slab. Bernstein and Breiner used Osserman-Schiffer's theorem [Theorem 1., \cite{osserman1975doubly}] on the convexity of the length of the circles foliating the minimal annulus. They also conjectured that the maximally symmetric marginally stable pieces of the catenoids in a slab achieve the least area among all smooth minimal surfaces spanning the same  slab. 
\medskip

Later, Choe and Daniel \cite{choe2016area} proved their conjecture under the additional assumption that the intersections of the minimal surface in the slab with horizontal planes have the same orientation. In the same work, they also proposed some related questions. Specifically, they consider an immersed minimal annulus with figure-eight level curves, denoted by $\Sigma$, within a horizontal slab, denoted by $\Omega$. They inquire whether the maximally symmetric and marginally stable waist of the catenoid in  $\Omega$ still attains the minimum area when compared to $\Sigma$, [Problem 2. \cite{choe2016area}]. We give an affirmative answer to this question for annuli with vertical flux in Theorem \ref{thm: area(k=0) > C}. 
\medskip

Fix two parallel planes $P_{\pm} \subset \mathbb{R}^3$ with $P_{+} \neq P_{-}$ and denote by $\Omega$ the open slab between them. By some rotation and rigid motions we may assume that $\Omega$ is parallel to $\{ x_3 =0 \}$, so that $P_+ = \{ x_3 = h_+ \} \ \text{and}\  P_- = \{ x_3 =h_- \}$, where $h_+ >0$ and $h_- = -h_+< 0$. Let $\mathcal{A} (\Omega )$ to be the set of all minimal annuli spanning $P_+$ and $P_-$. That is, $\Sigma \in \mathcal{A} (\Omega )$ if $\Sigma \subset \Omega$ may be parameterized by a proper conformal, harmonic immersion $\bold{F} : A_{R_0,R_1} \rightarrow \mathbb{R}^3$ so that $b\Sigma := \bar{\Sigma} \backslash \Sigma \subset \partial \Omega = P_+ \cup P_-$. Here, $A_{R_0,R_1} = \{ z\in \mathbb{C};\  R_0<\vert z\vert <R_1   \}$ is an open annular domain. 
\medskip

Recall that for a given pair of connected simple closed curves $\sigma_+ \subset P_+$ and $\sigma_- \subset P_-$, a necessary condition for the existence of a minimal surface in $\Omega$ spanning $\sigma_{\pm}$ is the existence of a waist of a catenoid in $\Omega$ spanning $P_{\pm}$ so that the length of its boundary in $P_{\pm}$ is less than or equal to the length of $\sigma_+ \cup \sigma_-$, see [Theorem 4.1, \cite{bernstein2014variational} and [Theorem 4, \cite{osserman1975doubly}]. And, for a given embedded annulus $\Sigma$ in $\Omega$ spanning $P_{\pm}$, the area of $\Sigma$ is greater than or equal to the area of the maximally symmetric marginally stable piece of a catenoid in $\Omega$ spanning $P_{\pm}$, and moreover, the two areas are equal if and only if $\Sigma$ is a translate of that waist of catenoid, [Theorem 1.1, \cite{bernstein2014variational}].
\medskip
 
Let
\begin{equation}
    \nonumber
    L(r) =\vert \bold{F}(\{ \vert z\vert =r\})\vert ,
\end{equation}
for $r \in (R_0,R_1)$, 
be the length of the image of the circle $\{ \vert z\vert =r \} \subset A_{R_0,R_1}$. The convexity result of Osserman-Schiffer's, \cite{osserman1975doubly}, implies that for $\bold{F}\in \mathcal{A}(\Omega )$ one has
\begin{equation}
    \nonumber
    L''(t) \geq L(t),
\end{equation}
where $t=\log r$.
\medskip

Denote by $\mathcal{A}_E(\Omega)$ the set of maps $\bold{F} \in \mathcal{A}(\Omega)$ that are embedding and have the domain $A_{R_0,R_1}$, so $\mathcal{A}_E(\Omega)$ is the set of embedded minimal annuli in a slab considered by Bernstein-Breiner, \cite{bernstein2014variational}, (called $\mathcal{A}(\Omega)$ in their paper).
They reparameterized $\Sigma \in \mathcal{A}_E(\Omega)$ in terms of the height, $h$, of the slab with $\bold{F} \in \mathcal{A}_E(\Omega )$, such that 
\begin{equation}
    \nonumber
    \{ x_3(\bold{F}(z))=h \} =\{ \vert z \vert =r_h \}.
\end{equation}
In fact, $\bold{F}$ reparameterizes $\Sigma$ with the vertical cylinder $(h_-,h_+)\times (\frac{F_3(\Sigma)}{2\pi})\mathbb{S}^1$, where  $F_3(\Sigma)$ is the third component of the flux vector of $\Sigma$, $Flux(\Sigma)=(F_1,F_2,F_3)$. In particular, they use the length of the flux vector as a natural scale to compare different minimal annuli in the reference set $\mathcal{A}(\Omega)$. The property of the flux that makes it an appropriate scale for a minimal annulus is that, 
$
Flux(r\Sigma)=rFlux(\Sigma)
$
for $r\Sigma$ a homothetical scale of $\Sigma$ by $r>0$.\\
To that end, we fix an orientation of $\Sigma$ and let $\sigma$ be a $C^1$ closed curve in $\Sigma$ on which we also fix an orientation. According to the choices of orientation, let $\nu$ be the conormal vector field in $\Sigma$ along $\sigma$, then 
$$
Flux(\sigma)= \int_{\sigma} \nu ds,
$$
By the fact that coordinate functions on $\Sigma$ are solutions to the Laplace equation and as an application of the divergence theorem, one can check that, the flux of a curve on $\Sigma$ only depends on its homology class. We associate a vector $Flux(\Sigma)$ to the surface $\Sigma$ by choosing $\sigma$ so that $[\sigma]$ is a generator of the first homology class of $\Sigma$, $H_1(\Sigma)$, and setting $Flux(\Sigma)=Flux(\sigma)$.
The reparameterization $\bold{F}$ of $\Sigma$ with the vertical cylinder $(h_-,h_+)\times (\frac{F_3(\Sigma)}{2\pi})\mathbb{S}^1$ is not depending on the surface $\Sigma$ being an embedding, and it applies to more general surfaces in $\mathcal{A}(\Omega)$.
\medskip

Denote the length of the level curves $\Sigma_h=\{ x_3(\bold{F}(z))=h \}$ foliating $\Sigma$ in $\bold{F}$ by 
\begin{equation}
    \nonumber
    \ell (h) = L(r_h). 
\end{equation}
In fact, $\ell(h)=\mathcal{H}^1(\Sigma_h)$, Hausdorff measure of $\Sigma_h$, for $h\in (h_-,h_+)$.\\
As observed by Bernstein-Breiner,
\begin{equation}
    \nonumber
    r_h=h_0 + \frac{2\pi}{F_3} h ,
\end{equation}
Hence, 
\begin{equation}  \label{ineq: ell''-k^2m^2ell}
    \ell ''(h) \geq \Big ( \frac{2\pi}{F_3} \Big )^2 \ell (h)
\end{equation}
One has equality only for a parameterization of a catenoid. 
\medskip

More interesting, they proved that: 
\medskip

\textit{``For any $\Sigma = \bold{F}(A_{R_0,R_1})$ where $\bold{F} \in \mathcal{A}_E(\Omega)$, let $P_0=\{x_3=h_0\}\subset \bar{\Omega}$ denote the plane that satisfies 
$$
\ell(h_0)=\underset{h\in (h_-,h_+)}{\inf}  \ell(h),
$$
where $\ell(h_+)$ is defined as $\underset{h\rightarrow h_+}{\lim} \inf \ell (h)$ and likewise for $\ell(h_-)$.
And, $F_3$ denote the vertical component of $Flux (\Sigma)$. If $C$ is the vertical catenoid with $Flux(C)=(0,0,F_3)$ and symmetric with respect to the reflection through the plane $P_0$ then:
$$
Area (\Sigma) \geq Area (C\cap \Omega)
$$
with equality if and only if $\Sigma$ is a translation of $C\cap \Omega$."}
\medskip

It follows from various properties of minimal immersions that the Gauss map of an immersion $\bold{F} \in \mathcal{A}(\Omega)$ gives rise to a holomorphic map
\begin{equation}
    \nonumber
    N_{\bold{F}} : A_{R_0,R_1} \rightarrow \mathbb{C} \backslash \{0\},
\end{equation}
and there is a well-defined winding number of this map which we denote by $WN(\bold{F}) \in \mathbb{Z}$. Note that the map $N_{\bold{F}}$ takes values
in $\mathbb{C} \backslash \{0\}$ because the surface $\bold{F}(A_{R_0,R_{1}})$ intersects horizontal planes
transversally.
We let
\begin{equation}  
\nonumber
\mathcal{A}(\Omega , k) = \Big\{ \bold{F} \in \mathcal{A} (\Omega)\ :\ WN(\bold{F}) = k \Big\},
\end{equation}
the set of minimal immersions of annuli in $\Omega$ for which the degree of the Gauss map $N$ is $k$. \\
Some simple observations are as follows: \\
(1) If $\bold{F} \in \mathcal{A}(\Omega)$ is an embedding, then $\bold{F} \in \mathcal{A}(\Omega ,\pm 1)$,\\
(2) If $P_l : \mathbb{C} \backslash \{0\} \rightarrow \mathbb{C} \backslash \{0\}$ is given by $P_l(z)=z^l$ for $l\neq 0$ and $\bold{F} \in \mathcal{A}(\Omega ,k)$, then $\bold{F} \circ P_l(z) \in \mathcal{A}(\Omega ,lk)$\\
(3) In particular, $\bold{F} \in \mathcal{A}(\Omega ,k)$ if and only if $\bold{F} \circ P_{-1} \in \mathcal{A}(\Omega ,-k)$,\\
(4) For each $k\neq 0$, there is a natural family of elements of $\mathcal{A}(\Omega ,k)$ given by $\bold{F}_C \circ P_k$ where $\bold{F}_C$ is a parameterization of a piece of a (vertical) catenoid.
\medskip

For $k\neq 0$, denote by 
\begin{equation}
    \nonumber
    \mathcal{A}^*(\Omega ,k) =\{ \bold{F} \in \mathcal{A}(\Omega ,k)\ :\ \bold{F} = \bold{G} \circ P_k,\ \bold{G} \in \mathcal{A}(\Omega ,1)\},
\end{equation}
the set of the $k$-fold covers of elements with winding number 1. Then,
\begin{equation}
    \nonumber
    \mathcal{A}^*(\Omega ,\pm 1) = \mathcal{A}(\Omega ,\pm 1). 
\end{equation}
But for $k\neq 1$ this is a proper subset, which means that $\mathcal{A}^*(\Omega ,\pm k) \subset \mathcal{A}(\Omega ,\pm k)$. Notice, the covers of parameterizations of catenoids lie in these spaces. 
\medskip

We identify $\mathcal{A}(\Omega ,-k)$ with $\mathcal{A}(\Omega ,k)$, for each integer $k$. 
\medskip

Let $\Sigma = \bold{F}(A_{R_0,R_1}) $ where $\bold{F} \in \mathcal{A}(\Omega,k)$ and $k> 1$, let $P_0=\{x_3=h_0\}\subset \bar{\Omega}$ denote the plane that satisfies 
$$
\ell(h_0)=\underset{h\in (h_-,h_+)}{\inf}  \ell(h),
$$
where $\ell(h_+)$ is defined as $\underset{h\rightarrow h_+}{\lim} \inf \ell (h)$ and likewise for $\ell(h_-)$.
And, $F_3$ denotes the vertical component of $Flux (\Sigma)$. For $k\neq 0$, let $C_k\in \mathcal{A}(\Omega,k)$ to be the $k$-fold cover of the vertical catenoid with $Flux(C_k)=(0,0,F_3)$ and symmetric with respect to the reflection through the plane $P_0$. We are concerned with the following question:
\medskip

\textit{Is still true that
$$
Area (\Sigma) \geq Area (C_k)
$$
with equality if and only if $\Sigma$ is a translation of $C_k$.}\\
The answer is no when $k$ is even, as shown in Theorem \ref{thm: main result}.
\medskip

The idea of checking their winding numbers comes from the work of Osserman-Schiffer. In the proof of their convexity result, [Theorem 1., \cite{osserman1975doubly}], they concluded by two cases for the minimal annuli based on their winding numbers, that is the winding number of the curves foliating the annulus. The first case is the set of all minimal annuli with even winding number, zero is included, and the second case is the set of all minimal annuli with odd winding number. \\
Furthermore, the essence of this note is based on the observation that self-intersections of the level curves of an annulus may allow for deformations that yield new minimal annuli with smaller areas. Specifically, fix the slab $\Omega$ to be horizontal, and choose an appropriate scale to compare and evaluate the minimal annuli in $\Omega$. The scale is defined as the length of the third component of the flux vector of the given vertical annulus. It is established that for any embedded annulus in $\mathcal{A}(\Omega)$, so the winding number is necessarily one, there is always a waist of the catenoid with the same scale in $\mathcal{A}(\Omega)$ such that the area of the annulus is greater than or equal to the area of the waist of the catenoid, see \cite{bernstein2014variational}. The equality holds only if the given annulus is the characterized waist of the catenoid. However, our findings show that if we consider level curves with even winding number $k>1$ then there may exist annuli other than $k$-fold covers of catenoids in $\mathcal{A}(\Omega ,k)$, which attain a smaller area than the $k$-fold covers of the waists of catenoids with the same scale. 
\medskip

In this writing, all the surfaces are in $\mathcal{A}(\Omega)$, particularly, they span $P_-$ and $P_+$. In particular, we  set\\
$\mathcal{A} (\Omega )=\{\bold{F}:A_{R_0,R_1} \rightarrow \Omega \ :\ \bold{F}\ \text{is a proper conformal minimal immersion} \} \\
\mathcal{A}_E (\Omega )=\{\bold{F} \in  \mathcal{A} (\Omega )\ :\ \bold{F}\ \text{ is an embedding} \}   \\
\mathcal{A}(\Omega , k) = \Big\{ \Sigma \in \mathcal{A} (\Omega)\ :\ \text{degree} (N) = k \Big\} \\
\mathcal{A}_V(\Omega,k)=\{ \bold{F} \in \mathcal{A}(\Omega,k)\ ;\ Flux(\bold{F})=(0,0,F_3)\}\\
\mathcal{H}^k\  \text{denotes}\ k\text{-dimensional Hausdorff measure} \\
\Sigma_h = \Sigma \cap \{ x_3 =h;\ h_-<h<h_+ \}\ \ \text{a level curve of}\ \Sigma \\
P_0 =\{ x_3=h_0\} \subset \bar{\Omega},\ \text{so that}\ \mathcal{H}^1 (\Sigma \cap P_0)= \underset{(h_-,h_+)}{inf}\mathcal{H}^1(\Sigma_h) \\ 
 C_{\Omega}\ \text{is the maximally symmetric marginally stable waist of the catenoid spanning $\Omega$} 
$ 

\section{Preliminaries}
\subsection{Background on the space $\mathcal{A}(\Omega,k)$}
In this subsection, we collect some results on the surfaces within the spaces $\mathcal{A}^*(\Omega,k)$ and $\mathcal{A}(\Omega,k)$. These results will help us to address the question of the minimality of the area of $k$-fold covers of the waists of catenoids in $\Omega$.

\begin{theorem} \label{thm: A^* least area}
    The least area element in $\mathcal{A}(\Omega , 1)$ is $C_{\Omega}$, the maximally symmetric marginally stable piece of the catenoid spanning $\Omega$.
\end{theorem}
\begin{proof}
   This result directly follows from [Theorem 1.1,\cite{bernstein2014variational}].
     
\end{proof}

\begin{theorem} \label{thm: A^* least area}
    The least area element in $\mathcal{A}^*(\Omega ,k)$ is a $k$-fold cover of catenoid, for $k \neq 0$.
\end{theorem}
\begin{proof}
    Let $\bold{F} \in \mathcal{A}^*(\Omega ,k)$. By the definition of  $\mathcal{A}^*(\Omega ,k) $, there is $\bold{G} \in \mathcal{A}(\Omega ,1)$ so that $\bold{F} = \bold{G} \circ P_k$, where $P_k : \mathbb{C} \backslash \{ 0\} \rightarrow \mathbb{C} \backslash \{ 0\}$ is given by $P_k(z)=z^k$. It is known that, $Area (\bold{F}(A_{R_0,R_1}))=\vert k\vert Area (\bold{G}(A_{R_0,R_1}))$. The least area element in $\mathcal{A}(\Omega ,1)$ is the maximally symmetric marginally stable piece of the catenoid in $\Omega$, denoted as $\bold{F}_{C_{\Omega}}$. It immediately follows that $\bold{F}_{C_{\Omega}} \circ P_k$ represents the least area element in the set $\{ \bold{G} \circ P_k\ :\ \bold{G} \in \mathcal{A}(\Omega ,1) \}$. 
\end{proof}

\begin{lem} \label{lem: ineq. lenghth for A^*}
    For $\bold{F} \in \mathcal{A}^*(\Omega ,k)$,
    $$
    L''(t)\geq k^2 L(t),
    $$
   or equivalently, 
\begin{equation} 
\nonumber
    \ell ''(h) \geq k^2 \Big ( \frac{2\pi}{F_3} \Big )^2 \ell (h)
\end{equation}
with equality only when $\bold{F}$ parameterizes a $k$-fold cover of the catenoid.
\end{lem}
\begin{proof}
    Recall that associated with $\bold{F}$ there is $\bold{G}\in \mathcal{A}(\Omega,1)$ so that $\bold{F}=\bold{G}\circ P_k$, where $P_k:\mathbb{C}\backslash \{0\} \rightarrow \mathbb{C}\backslash \{0\}$ is given by $P_k(z)=z^k$. We check that $\ell(h)=\vert \bold{F}(\vert z\vert =e^{\frac{h}{k}}) \vert $, so the result follows from a simple computation. The last statement is a direct result of Theorem \ref{thm: A^* least area}.
\end{proof}
\medskip

Let $\bold{F} \in \mathcal{A}(\Omega ,k)$, for $k=1,2, \cdots$. Note that if the surface $\Sigma = \bold{F}(A_{R_0,R_1})$ is symmetric with respect to reflection through a horizontal plane $\{ x_3=h,\ h\in (h_-,h_+)\}$, then the flux is vertical, which means that $Flux(\Sigma)=(0,0,F_3)$. In the rest of this writing, we pertain to a simplifying hypothesis that the flux of the surface $\Sigma = \bold{F}(A_{R_0,R_1})$ is vertical, and we denote it by $\mathcal{A}_V(\Omega,k)=\{ \bold{F} \in \mathcal{A}(\Omega,k)\ ;\ Flux(\bold{F})=(0,0,F_3)\}$.

\subsection{Geometric background}
Let $\Sigma$ be an annulus parameterized in $ \mathcal{A}_V(\Omega)$, and denote its minimal conformal parameterization by $\bold{X}(z)=(X_1(z),X_2(z),X_3(z)) : A_{R_0,R_1} \rightarrow \mathbb{R}^3$. Here, the functions $X_k$ are harmonic in $z$, and the functions $\phi_k := \frac{\partial X_k}{\partial x} - i\frac{\partial X_k}{\partial y}$ are analytic in $z:=x+iy$, for $k=1,2,3$. The element of arc length $ds$ on the surface $\Sigma$ is given by
\begin{equation}
    \nonumber
    ds^2 = \lambda^2 \vert dz\vert^2,\ \ \text{where}\ \lambda = \sqrt{\frac{1}{2}\sum_{i=1}^{3}\vert \phi_i \vert^2}.
\end{equation}
And,
\begin{equation}
    \nonumber
    L(r) =\vert \bold{X}(\{ \vert z\vert =r\})\vert ,
\end{equation}
is the length of the image of the circle $\{ \vert z\vert =r \} \subset A_{R_0,R_1}$. Let $\psi_i = z\phi_i$, and define $\mu =r\lambda = \sqrt{\frac{1}{2}\sum_{i=1}^{3}\vert \psi_i \vert^2}$, then we may write 
$$
L(r)= \int_{0}^{2\pi}\mu (re^{i\theta})d\theta  $$

Recall that, for $\bold{X}\in \mathcal{A}_V(\Omega )$ one has
\begin{equation}
    \nonumber
    L''(t) \geq L(t),\ \text{for}\  t=\log r
\end{equation}

For a general minimal annulus $\bold{X}(A_{R_0,R_1}) \in \mathcal{A}_V(\Omega)$, the Weierstrass representation of $\bold{X}$ in $\mathbb{R}^3$ can be expressed as
$$
\bold{X}(p)=Re \int_0^{2\pi} \Big( \frac{1}{2}(g^{-1}-g), \frac{i}{2}(g^{-1}+g),1\Big)dh
$$
where $g$ is the extended Gauss map, a meromorphic function defined on its Riemann surface, $A_{R_0,R_1}$, and $dh$ is a holomorphic 1-form on $A_{R_0,R_1}$. Note that,
$$
\bold{X}(p)=Re\int \bold{\Phi} dz
$$
where $\bold{X}(z) =(X_1(z),X_2(z),X_3(z)),\ \bold{\Phi} =(\phi_1,\phi_2,\phi_3)$, and $\phi_i =\frac{\partial X_i}{\partial x}- i\frac{\partial X_i}{\partial y}$, for $z=x+iy$. 

$\bold{X}$ is well defined on its domain if and only if no component of $\bold{\Phi}$ has a real period. That is
$$
 \text{Period}_{\gamma}(\bold{\Phi})(=\text{P}(\gamma)):= Re \oint_{\gamma} \bold{\Phi} dz=0
$$
for all closed curves $\gamma$ on $A_{R_0,R_1}$: this is the Period Problem. A good reference on Weiestrass representation is \cite{hoffman1997complete}. 
We express
\begin{equation} \label{phis}
    \phi_1 =\frac{1}{2}f(1-g^2), \  \phi_2=\frac{i}{2}f(1+g^2), \  \phi_3 =fg
\end{equation}
where $dh=fgdz$, $f$ is analytic in $A_{R_0,R_1}$, the Gauss map $g$ is meromorphic in $A_{R_0,R_1}$, $g$ has poles precisely at the zeros of $f$, and each zero of $f$ has an order equal to twice the order of the pole of $g$ at that point. Then
\begin{equation}
    \lambda = \frac{1}{2} \vert f\vert (1+\vert g\vert^2) =\frac{1}{2} \vert \phi_3 \vert (\frac{1}{\vert g\vert} +\vert g\vert )
\end{equation}
and
\begin{equation} \label{eqn: mu}
    \mu = \frac{1}{2} \vert \psi_3 \vert (\frac{1}{\vert g\vert} +\vert g\vert )
\end{equation}
Recall that $\psi_i = z\phi_i$, for $i=1,2,3$. We can compute 
\begin{equation} \label{def. psi3}
     \begin{split}
    \psi_1=\frac{1}{2}\psi_3 (\frac{1}{g}-g), \ \  \ \psi_2=\frac{i}{2}\psi_3 (\frac{1}{g}+g) \\
    \frac{\psi_3}{g} = \psi_1- i\psi_2, \  \  \  \psi_3 g = -\psi_1 -i\psi_2 .
    \end{split}
\end{equation}

\begin{lem} \label{lem: good repn}
    For any annulus $\bold{X}: A_{R_0,R_1} \rightarrow \Omega$ in $\mathcal{A}_V(\Omega)$, up to homothety, we have $dh = \frac{F_3}{2\pi} \frac{dz}{z}$.
\end{lem}
This is proved in \cite{bernstein2014variational}, see Corollaries 2.2 and 2.4., their proof is based on the fact that the embedded annuli in $\mathcal{A}_V(\Omega)$ are transverse to horizontal planes, Lemma 2.1 \cite{bernstein2014variational}. However, the transversally result holds for our immersed surfaces $\bold{X} \in \mathcal{A}_V(\Omega)$ since the height function is harmonic and constant on boundary components, so it cannot have any critical point on the interior.

\begin{rem} \label{rmk: good repn}
    \rm  For the above reparameterization of the surface $\bold{X}$, the complexified height differential is $dh=c\frac{dz}{z}$. That gives the map $\psi_3=z\phi_3=\frac{zdh}{dz}=c$ is constant.
\end{rem}

Osserman-Schiffer \cite{osserman1975doubly} observed that the length functions $L''(t)$ and $L(t)$ can be derived from the two functions $\frac{\psi_3}{g}$ and $\psi_3g$, enabling the derivation of certain relationships between $L''(t)$ and $L(t)$. More specifically, they noted that 
\begin{equation}\label{L"}
    \frac{d^2L}{dt^2} = \frac{1}{2} \int_{0}^{2\pi} \big( r^2\Delta \vert \frac{\psi_3}{g} \vert + r^2\Delta \vert \psi_3 g\vert  \big) d\theta 
\end{equation} 
and 
\begin{equation}\label{L}
    L=\frac{1}{2} \int_{0}^{2\pi} \big( \vert \frac{\psi_3}{g} \vert +  \vert \psi_3 g\vert  \big) d\theta .
\end{equation}

Furthermore, they assumed that the functions $\frac{\psi_3}{g}$ and $ \psi_3 g$ satisfy:
\begin{equation} \label{eqn: int, pp}
    \int_{0}^{2\pi} \frac{\psi_3}{g} (re^{i\theta})d\theta = 0, \  \ \int_{0}^{2\pi} \psi_3 g (re^{i\theta})d\theta = 0.
\end{equation}

In essence, solving the two integral equations above is equivalent to addressing the period problem for the parameterization of $\bold{X}$, as long as the flux is vertical. Consequently, the representation $\bold{X}$ associated with the surface $\Sigma=\bold{X}(A_{R_0,R_1})$ is considered well-defined on $A_{R_0,R_1}$ if it satisfies these two integral equations. We summarize the above:

\begin{proposition}
    The provided Weierstrass representation data for $\bold{X}$ satisfies the period problem, $P(\gamma)=0$ for any closed cycle on $A_{R_0,R_1}$, and (\ref{eqn: int, pp}) holds.  
\end{proposition}

Setting $F_-(z) = \frac{\psi_3}{g} $  or $F_+(z) = \psi_3 g $, the functions $F_{\pm}(z)$ are analytic on $A_{R_0,R_1}$ and satisfy:

\begin{equation} \label{eqn: int F =0}
\int_{0}^{2\pi} F_{\pm}(re^{i\theta} )d\theta =0 
\end{equation}
Assuming that $F_{\pm}$ are non-zero on $A_{R_0,R_1}$, we have two possibilities:
\begin{equation}
    \text{Case 1.}\ F_{\pm}(z)=G_{\pm}^2(z),\ \ \ \text{or}\ \ \ \text{Case 2.}\ F_{\pm}(z)=zG_{\pm}^2(z),
\end{equation}  
where $G_{\pm}(z)$ are analytic in $A_{R_0,R_1}$ and have there the Laurent expansions:
\begin{equation}
    \nonumber
    G_{\pm}(z) = \underset{n=-\infty }{\overset{\infty}{\sum}} a_n z^n
\end{equation}

Note that, 
\begin{equation}\label{eqn: int-arg=2pik}
    \int_{0}^{2\pi} \frac{\partial}{\partial \theta} \text{arg} F_{\pm}(re^{i\theta}) d\theta =2\pi k_{\pm},
\end{equation}
where $k_{\pm}$ denote the winding number of the circles that foliate the annular domain of $F_{\pm}$.
These two cases, 1 and 2, correspond to the winding number $k_{\pm}$ being even and odd for the circles that foliate the annulus $A_{R_0,R_1}$, respectively. 
Note that, by the fact that $\psi_3=c$ we can conclude that if $F_+$ satisfies case 1 (resp. case 2) then $F_-$ satisfies case 1 (resp. case 2) as well. 
\medskip

\begin{proposition}
    The winding number of the Gauss map, $g$, is equal to $k_+=-k_-$, in (\ref{eqn: int-arg=2pik}).
\end{proposition}
\begin{proof} We may reparametrize $\bold{X}$ as in Lemma \ref{lem: good repn}, to conclude that the winding number of the map $\psi_3$ is zero, see Remark \ref{rmk: good repn}. Then, for $F_+=\psi_3 g$ we obtain
    $$
    \int_{0}^{2\pi} \frac{\partial}{\partial \theta} \text{arg} F_+(re^{i\theta}) d\theta = \int_{0}^{2\pi} \frac{\partial}{\partial \theta} \text{arg} (\psi_3 (re^{i\theta})g(re^{i\theta})) d\theta \\ 
    = \int_{0}^{2\pi} \frac{\partial}{\partial \theta} \text{arg} (g(re^{i\theta})) d\theta , 
    $$ 
   where $g$ is the Gauss map of the surface represented by $\bold{X}$.
    For $F_-=\frac{\psi_3}{g}$, a similar argument gives the result by a negative sign.
\end{proof}
\medskip

For the analytic function $G_{\pm}(z) = \underset{n=-\infty }{\overset{\infty}{\sum}} a_n z^n$ in $A_{R_0,R_1}$ we have: \\
\begin{equation}
\nonumber
\begin{split}
& \int_{0}^{2\pi} \vert G_{\pm}(re^{i\theta} )\vert^2 d\theta = 2\pi \underset{n=-\infty }{\overset{\infty}{\sum}} \vert a_n \vert^2 r^{2n} , \\
& \int_{0}^{2\pi} \vert G'_{\pm}(re^{i\theta} )\vert^2 d\theta = 2\pi \underset{n=-\infty }{\overset{\infty}{\sum}} n^2 \vert  a_n \vert^2 r^{2n-2} , \\
& \Delta \vert G_{\pm}(z)\vert^2 =4\frac{\partial^2}{\partial z\partial \bar{z}} G_{\pm}(z) \overline{G_{\pm}(z)} = 4\vert G'_{\pm}(z)\vert^2 ,\\
& \int_{0}^{2\pi} r^2 \vert G'_{\pm} \vert^2 d\theta \geq \int_{0}^{2\pi} \vert G_{\pm}\vert^2 d\theta -2\pi \vert a_0\vert^2 ,
\end{split}
\end{equation}

and

\begin{equation} \label{ineq: int. norm.sqr.G}
    \int_{0}^{2\pi} r^2 \Delta \vert G_{\pm}\vert^2 d\theta \geq 4 \int_{0}^{2\pi} \vert G_{\pm}\vert^2 d\theta -8\pi \vert a_0\vert^2.
\end{equation}

\begin{rem}
   \rm It is easy to check that in (\ref{ineq: int. norm.sqr.G}) the equality happens, for all $r$, if and only if $a_n=0$ for all $n\neq -1,0,1$.
\end{rem}

 Throughout the remainder of this work, when referring to the functions $G_{\pm}(z)$, $F_{\pm}(z)$, $g(z)$, $f(z)$, $\phi_i(z)$, and $\psi_i(z)$, we specifically discuss the functions detailed earlier and defined on the annulus $A_{R_0,R_1}$ to parameterize the given surface.

\section{Main discussion and Results}
 In this note, we focus mainly on the annuli exhibiting an even winding number. 

\begin{lem} \label{lem: L''(t)>2L(t)}
     For surfaces in $\mathcal{A}_V(\Omega ,2k)$, the relationship  $L''(t) > 2L(t)$ holds, where $k=0,1,2,\cdots  $.
\end{lem}
\begin{proof}
   Let $\bold{X} \in \mathcal{A}_V(\Omega, 2k)$, for $k=0,1,2,\cdots$. In this case, $F_{\pm}(z)=G_{\pm}^2(z)$. Using (\ref{eqn: int, pp}), it is evident that the constant term in the Laurent expansion of $F_{\pm}(z)$, denoted as  $a_0^2 + 2\underset{n=1 }{\overset{\infty}{\sum}} a_n a_{-n}$, must be equal to zero:
\begin{equation}
    \nonumber
    a_0^2 + 2 \underset{n=1 }{\overset{\infty}{\sum}} a_n a_{-n} =0.
\end{equation}
Compute
\begin{equation}
    \nonumber
    \begin{split}
    \vert a_0 \vert^2 & = 2 \Big\vert \underset{n=1 }{\overset{\infty}{\sum}} a_n r^n a_{-n} r^{-n} \Big\vert \leq \underset{n=1 }{\overset{\infty}{\sum}} (\vert a_n \vert^2 r^{2n} + \vert a_{-n} \vert^2 r^{-2n} ) \\
    & = \frac{1}{2\pi} \int_{0}^{2\pi} \vert G_{\pm}(re^{i\theta} ) \vert^2 d\theta - \vert a_0\vert^2
    \end{split}
\end{equation}
Hence,
\begin{equation} \label{ineq: a_0&int.norm.sqrG}
    4\pi \vert a_0\vert^2 \leq \int_{0}^{2\pi} \vert G_{\pm}\vert^2 d\theta 
\end{equation}
By substituting this into inequality (\ref{ineq: int. norm.sqr.G}), and considering $\vert F_{\pm}\vert =\vert G_{\pm}\vert^2$, we obtain:
\begin{equation}  \label{ineq: intg. norm F}
\int_{0}^{2\pi} r^2 \Delta \vert F_{\pm}\vert d\theta \geq 2\int_{0}^{2\pi} \vert F_{\pm}\vert d\theta  
\end{equation}
Equality holds, for any $r\in (R_0,R_1)$, if and only if  
$$
 4\pi \vert a_0\vert^2 = \int_{0}^{2\pi} \vert G_{\pm}\vert^2 d\theta ,
$$
where $\int_{0}^{2\pi} \vert G_{\pm}\vert^2 d\theta = 2\pi \underset{n=-\infty }{\overset{\infty}{\sum}} \vert a_n \vert^2 r^{2n}$. Hence, $\underset{n=-\infty }{\overset{\infty}{\sum}} \vert a_n \vert^2 r^{2n}= 2 \vert a_0\vert^2$, for $r\in (R_0,R_1)$. Note that, $\underset{n<0}{\sum} \vert a_n \vert^2 z^{2n}$ defines a holomorphic function on $\mathbb{C}\backslash D_{R_0}(0)$ while $\underset{n\geq 0 }{\sum} \vert a_n \vert^2 z^{2n}$ defines a holomorphic function on the disk $D_{R_1}(0)$. Notice that the function $\underset{n\geq 0 }{\sum} \vert a_n \vert^2 z^{2n}$ coincide with $2 \vert a_0\vert^2 -\underset{n<0}{\sum} \vert a_n \vert^2 z^{2n}$ on the annulus $A_{R_0,R_1}$, so we may 
consider the following function:
$$
f(z)=   
\begin{cases}
&\underset{n\geq 0 }{\sum} \vert a_n \vert^2 z^{2n} \hspace{1.6cm} \text{on}\  D_{R_1} \\
&2 \vert a_0\vert^2 -\underset{n<0}{\sum} \vert a_n \vert^2 z^{2n}\ \   \text{on}\  \mathbb{C}\backslash D_{R_0} 
\end{cases} 
$$
$f(z)$ is entire and bounded, thus it is constant. It follows that both $\underset{n<0}{\sum} \vert a_n \vert^2 r^{2n}$ and $\underset{n\geq 0 }{\sum} \vert a_n \vert^2 r^{2n}$ are constant on the interval $(R_0,R_1)$. However, this is not possible for non-constant $G_{\pm}(z)$, when $R_0\neq R_1$. 
We conclude that, $\int_{0}^{2\pi} r^2 \Delta \vert F_{\pm}\vert d\theta > 2\int_{0}^{2\pi} \vert F_{\pm}\vert d\theta $, which is equivalent to $L''(t) > 2L(t)$.
\end{proof}

\begin{rem}
    \rm  If $\Sigma \in \mathcal{A}_V^*
(\Omega ,k)$ for $k$ being an even number, then according to Theorem \ref{thm: A^* least area}, the least area surface is a $k$-fold cover of a catenoid. Furthermore, as per Lemma \ref{lem: ineq. lenghth for A^*}, $L''(t) \geq k^2 L(t)$ with equality for all $t$ only happens for the surface be a $k$-fold cover of the catenoid. Specifically, for all surfaces parameterized in $\mathcal{A}_V^*
(\Omega ,2k),\ k\neq 0$, we have $L''(t) \geq 4L(t)$ with equality, on the domain $A_{R_0,R_1}$, only for 2-fold covers of a waist of a catenoid.
\end{rem}

Set $F_-=\frac{\psi_3}{g}$ and $F_+=\psi_3 g$, and subsequently $G_-^2(z)=F_-(z)$ or $zG_-^2(z)=F_-(z)$, and $G_+^2(z)=F_+(z)$ or $zG_+^2(z)=F_+(z)$ depending on the winding number of the level curves being even or odd, respectively.

\begin{lem} \label{lem: A* implies int G =0}
    If $\bold{X} \in \mathcal{A}_V^*(\Omega ,2k)$ for $k=1,2,3,\cdots$, then 
    $$
    \int_0^{2\pi} G_-(z)d\theta =0,\ \ \text{and}\  \int_0^{2\pi} G_+(z)d\theta =0.
    $$ 
\end{lem}
\begin{proof}
    Suppose that $\bold{X} \in \mathcal{A}_V^*(\Omega ,2k)$ parameterizes a $2k$-cover, $\Sigma$, of the surface $\widetilde{\Sigma}$ with the corresponding parameterization $\widetilde{\bold{X}} \in \mathcal{A}_V(\Omega ,1)$. The functions $G_{\pm}(z), F_{\pm}(z)$, $\widetilde{G_{\pm}}(z)$, and $ \widetilde{F_{\pm}}(z)$ correspond to the surfaces $\Sigma$ and $\widetilde{\Sigma}$. Recall that
    $$
    F_{\pm}(z)=G_{\pm}^2(z)  
    $$
    while, 
    $$
    \widetilde{F_{\pm}}(z)=z \widetilde{G_{\pm}}^2(z),
    $$
    where $F_-(z)=\frac{\psi_3}{g},\ F_+(z)=\psi_3g$, and similar equations for $\widetilde{F_{\pm}}(z)$. 
     \medskip
     
     Explicitly, 
     $
     g(z)=\widetilde{g}(z^{2k})$, and, $F_{\pm}(z)=2k\widetilde{F}_{\pm}(z^{2k})=2kz^{2k}\widetilde{G}^2_{\pm}(z^{2k})$.\\
   Let $\underset{n=-\infty}{\overset{\infty}{\sum}}  a_n z^n$ and $\underset{n=-\infty}{\overset{\infty}{\sum}}  \widetilde{a}_n z^n$ be the Laurent expansions of $G_{\pm}(z)$ and $\widetilde{G_{\pm}}(z)$, respectively.\\
   Recall that 
    $$
    \int_0^{2\pi}  G_{\pm}(z) d\theta =0
    $$
    if and only if $\vert a_0\vert^2 =0$. \\
    Note that, $G_{\pm}(z)=\sqrt{2k}z^k\widetilde{G}_{\pm}(z^{2k})=\sqrt{2k} \underset{n=-\infty}{\overset{\infty}{\sum}}  \widetilde{a}_n z^{(2n+1)k}$, and thus $a_0=0$ in the Laurent series of $G_{\pm}(z)$. 
    \end{proof}

    \begin{rem}
        \rm
    Recall that, by equation (\ref{L}) we have for $L(t)$ (and similarly for $\widetilde{L}(t)$):
    $$
    L(t)=\frac{1}{2}\int_0^{2\pi} (\vert F_-\vert +\vert F_+\vert )d\theta ,
    $$
   Moreover,
    $$
    \int_0^{2\pi}\vert F_{\pm}\vert d\theta =2k \int_0^{2\pi} \vert \widetilde{F_{\pm}} \vert d\theta
    $$
    That is equivalent to
    $$
    \int_0^{2\pi}\vert G_{\pm}\vert^2 d\theta =2k\rho \int_0^{2\pi} \vert \widetilde{G_{\pm}} \vert^2 d\theta ,
    $$
    for any $\rho \in (R_0,R_1)$, where the surfaces $\Sigma$ and $\widetilde{\Sigma}$ are two conformal minimal parameterizations of the annulus $A_{R_0,R_1}$. 
    Thus, $\Sigma$ is a $2k$-cover of $\widetilde{\Sigma}$, implying that $ L(t)=2k\widetilde{L}(t)$. 
 \end{rem}

\begin{lem}
    For a surface $\bold{X} \in \mathcal{A}_V(\Omega ,2)$ if $\int_0^{2\pi} G_-(z)d\theta = c_- $ and $ \int_0^{2\pi} G_+(z)d\theta = c_+$, then  
    $$
    L''(t) \geq 4L(t) -  \frac{1}{\pi}(\vert c_-\vert^2 +\vert c_+\vert^2 )
    $$
    In particular, if $c_-=c_+=0$ then
    $$
    L''(t) \geq 4L(t)
    $$
    and equality holds for all $t$ if and only if $\bold{X}$ is a 2-fold cover of a waist of a catenoid
\end{lem}
\begin{proof}
    Note that $\int_0^{2\pi} G_{\pm}(z) d\theta =2\pi a_0$, where $a_0$ is the constant term in the Laurent expansion of $G_{\pm}$. By applying the inequality (\ref{ineq: int. norm.sqr.G}), we derive the following inequality:
    $$
     \int_{0}^{2\pi} r^2 \Delta \vert G_{\pm} \vert^2 d\theta \geq 4 \int_{0}^{2\pi} \vert G_{\pm} \vert^2 d\theta - \frac{2}{\pi} \vert c_{\pm} \vert^2 
    $$  
    Recall the equations
    $$
    L''(t)=\frac{1}{2}\int_0^{2\pi}r^2 (\Delta \vert G_-\vert^2 + \Delta \vert G_+\vert^2)
    $$
    and,
    $$
    L(t) = \frac{1}{2}\int_0^{2\pi} ( \vert G_-\vert^2 +  \vert G_+\vert^2)
    $$
    This can be expressed as
    $$
    L''(t) \geq 4L(t) - \frac{1}{\pi}(\vert c_-\vert^2 +\vert c_+\vert^2 )
    $$
    In the second part, by Lemma \ref{lem: ineq. lenghth for A^*} the equality holds only for 2-fold covers of waist of catenoids.
\end{proof}

\begin{theorem} \label{thm: existence for L''<4L}
    There exist minimal surfaces in $\mathcal{A}_V(\Omega,2)\backslash \mathcal{A}^*(\Omega,2)$ for which
    $$
    L''(t)<4L(t),\ \ \forall t
    $$
\end{theorem}
\begin{proof}
    We identify examples of such surfaces by introducing perturbations to the extended Gauss map of a 2-fold cover of a catenoid.\\
 Define the slab $\Omega$ as outlined in the introduction, where $\Omega$ represents the open space between two parallel planes, denoted as $P_{\pm} \subset \mathbb{R}^3$, with $P_{\pm} =\{ x_3 =h_{\pm} \}$. Here, $h_-=-h_+$ with $h_+$ being a positive number. The surface $\bold{F} \in \mathcal{A}_V(\Omega ,2)$ is parameterized by the proper conformal, harmonic immersion $\bold{F}:A_{R_0,R_1} \rightarrow \mathbb{R}^3$, where $A_{R_0,R_1}=\{ z\in \mathbb{C} ;\  R_0< \vert z\vert <R_1 \}$ is an open annular domain. Moreover, $\bold{F}(A_{R_0,R_1}) \subset \Omega$ and $\bar{\Sigma} \backslash \Sigma \subset \partial \Omega$, where $\partial \Omega =  P_+ \cup P_- $. \\
Let $C_2$ be the 2-fold cover of a waist of catenoid parameterized by $\bold{F} \in \mathcal{A}_V(\Omega , 2)$, and assume the following Weierstrass data corresponding to $C_2$:
     $$
     g(z)=\frac{d_1}{z^2} \ \text{and}\ f(z)=d_2z,
     $$
     the corresponding representation is
     $$
     \begin{cases}
         \phi_1 = \frac{d_2z}{2} (1-\frac{d_1^2}{z^4})=\frac{d_2}{2z^3}(z^4-d_1^2) \\
         \phi_2 = \frac{id_2z}{2} (1+\frac{d_1^2}{z^4})=\frac{id_2}{2z^3}(z^4+d_1^2)  \\
         \phi_3 = \frac{d_1d_2}{z}
     \end{cases}
     $$
moreover, $F_-(z)=\frac{\psi_3}{g}=c_1z^2$ and $F_+(z)=\psi_3 g=c_2z^{-2}$.
Consider a perturbation of the Gauss map of $C_2$ given by
    $$
    g(z)=\frac{c_2 z^{-1}+\epsilon_2 +\delta_2 z}{c_1z+\epsilon_1 + \delta_1 z^{-1}},
    $$
   where $\epsilon_1, \delta_1, \epsilon_2, \delta_2$ are sufficiently small numbers, and they satisfy $\epsilon_1^2+2\delta_1c_1=0$ and $\epsilon_2^2+2\delta_2 c_2=0$. 
   By choosing these numbers sufficiently small, we can adjust the annular domain $A_{R_0,R_1}$ so that within $A_{R_0,R_1}$, the Gauss map remains free of poles and zeros. Accordingly,
   $$  f(z)=\frac{c_1^2z^2+\delta_1^2z^{-2}+2\epsilon_1 \delta_1 z^{-1}+2c_1\epsilon_1z}{z}
    $$
We compute $\psi_3 =z\phi_3$, where $\phi_3=fg$,
$$
\psi_3 (z) = (c_1z+\epsilon_1 + \delta_1 z^{-1}) (c_2 z^{-1}+\epsilon_2 +\delta_2 z)
$$
Given $g(z)$ set
    $$
    G_-(z)= c_1z+\epsilon_1 + \delta_1 z^{-1},\ \ \  G_+(z)=c_2 z^{-1}+\epsilon_2 +\delta_2 z 
    $$
    The computation yields:
    $$
    F_-(z)=G_-^2(z)=c_1^2z^2+\delta_1^2z^{-2}+2\epsilon_1 \delta_1 z^{-1}+2c_1\epsilon_1z,
    $$
   
    $$
    F_+(z)=G_+^2(z)=c_2^2z^{-2}+\delta_2^2z^2+2\epsilon_2 \delta_2 z+2c_2\epsilon_2z^{-1}
    $$
    Note that the mappings $F_-$ and $F_+$ satisfy
    $$
    \int_0^{2\pi} F_- d\theta =0,\ \int_0^{2\pi} F_+ d\theta =0
    $$
    This ensures that the new data results in a well-defined map on the adjusted domain $A_{R_0,R_1}$, and it also implies that the new surface maintains a vertical flux.
   \medskip
   
    The mappings $g(z)$ and $f(z)$ provide the parameters for the following Weierstrass representation:
    $$
    \bold{X}(p) = Re \int_{z,z_0} (\phi_1,\phi_2,\phi_3)dz
    $$
    where,
    $$
    \begin{cases}
        \phi_1=\frac{f}{2}(1-g^2) \\
        \phi_2=\frac{if}{2} (1+g^2) \\
        \phi_3 =fg
    \end{cases}
    $$
    Recall that the mapping $\bold{X}(p)$ is well-defined on $A_{R_0,R_1}$. 
    \medskip
    
    Let $\widetilde{\Sigma} = \bold{X} (A_{R_0,R_1})  \subset \mathbb{R}^3$
    represent the minimal surface in $\mathbb{R}^3$ parameterized by $\bold{X}$. Now, let $\Omega'$ be an open slab delimited by two parallel planes $P'_{\pm}$, where $P'_{\pm}=\{ x_3=h'_{\pm} \}$. Choose $0<h'_+ < h_+$ as a small number and $h'_-=-h'_+$, consequently, $\Omega' \subset \Omega$. For $h'_+$  sufficiently small, the associated Gauss map of $\bold{X}$ possesses a well-defined winding number, indicative of the number of times the circles foliating the surface $\widetilde{\Sigma}$ wind around it.
    We can confirm that this winding number remains two by directly evaluating the integral $\int_0^{2\pi} \frac{\partial}{\partial \theta} arg F(re^{i\theta}) d\theta$. Here, $\epsilon_j$ and $\delta_j$ are chosen to be sufficiently small, ensuring that the integral can be approximated by $\int_0^{2\pi} \frac{\partial}{\partial \theta} arg (c_j^2r^2e^{2i\theta}) d\theta$, for $j=1,2$. \\
  Define
$$
\Sigma = \widetilde{\Sigma} \cap \Omega'
$$
Subsequently,
$$
\bold{X}^{-1}(\Sigma) = A
$$
for some topological annulus $A$. By the Riemann mapping theorem there exists a conformal diffeomorphism $T$ from some planar annulus $A_{R_2,R_3}$ onto $A$, for certain positive values $R_2<R_3$, and $A_{R_2,R_3} \subset A_{R_0,R_1}$. Consider the map $\bold{X}\circ T: A_{R_2,R_3}\rightarrow \mathcal{A}_V(\Omega',2)$, and note that for a sufficiently thin slab $\Omega'$, and subsequently, a thin annulus $A$, the map $T$ is close to identity. Hence, $\bold{X}\circ T :=\hat{\bold{X}}$ has the Weierstrass data that approximates that of $\bold{X}$, thus $\hat{\bold{X}}=\bold{X}\circ T(A_{R_2,R_3})\in \mathcal{A}_V(\Omega',2)$.  
Observing that, 
$$
\int_0^{2\pi} \hat{G}_-(z)d\theta =2\pi \hat{\epsilon}_1 \neq 0,
$$
$$
\int_0^{2\pi} \hat{G}_+(z)d\theta =2\pi \hat{\epsilon}_2 \neq 0,
$$
we can conclude from Lemma \ref{lem: A* implies int G =0} that $\hat{\bold{X}} \in \mathcal{A}_V(\Omega',2) \backslash \mathcal{A}^*(\Omega',2)$.\\
Finally, we verify that
$$
\begin{cases}
    \int_0^{2\pi}r^2\Delta \vert \hat{G}_-\vert^2 d\theta =4\int_0^{2\pi} \vert \hat{G}_-\vert^2 d\theta -8\pi \vert \hat{\epsilon}_1 \vert^2 \\
    \int_0^{2\pi}r^2\Delta \vert \hat{G}_+\vert^2 d\theta =4\int_0^{2\pi} \vert \hat{G}_+\vert^2 d\theta -8\pi \vert \hat{\epsilon}_2 \vert^2
\end{cases}
$$
 Consequently,
 $$
 L''(t)=4L(t)-8\pi \vert \hat{\epsilon} \vert^2
 $$
 Given that $\epsilon$ is non-zero, it follows that 
 $$ L''(t)<4L(t),\ \forall t.$$
\end{proof}

\begin{proposition} \label{prop: symm.surface with L''<L}
    For the surface described in Theorem \ref{thm: existence for L''<4L}, we may choose the constants $c_1,\epsilon_1, \delta_1, c_2, \epsilon_2, \delta_2$ in such a way that the surface exhibits symmetry with respect to reflection through the horizontal plane $\{ x_3=0\} \subset \Omega$. 
\end{proposition}
\begin{proof}
    We assume that the surface $\Sigma$ is represented by the mapping 
    $$
    \bold{X}(z)=\int_{p_0}^p\Big( \frac{1}{2}(g^{-1}-g),\frac{i}{2}(g^{-1}+g),1\Big)dh
    $$
    where $\bold{X}(z):A_{R_0,R_1}\rightarrow \mathbb{R}^3$ is well-defined on the annular domain $A_{R_0,R_1}$ for $0<R_0<1$ and $1<R_1$, and  $dh=fgdz$, for $f$ and $g$ as provided in Theorem \ref{thm: existence for L''<4L}.  \\
    Hence, $\bold{X}(A_{R_0,R_1})=\Sigma$ as constructed in Theorem \ref{thm: existence for L''<4L} lies in $ \mathcal{A}_V(\Omega, 2)$, for some open slab $\Omega$ satisfying the properties defined for $\Omega'$. \\
   Define 
    $$
    \psi(z):A_{R_0,R_1}\rightarrow A_{R_0,R_1}, 
    $$
     by
    $$
     \psi(z)=\frac{1}{\bar{z}}
    $$
    Note that $\psi$ is identity on $\vert z\vert =1$.\\
    Consider the pullback of the maps $g$ and $dh$ by the map $\psi$
    $$
    \psi^*g=g\circ \psi =g(\frac{1}{\bar{z}})
    $$
  and
  $$
  \psi^*dh=dh\circ \psi = dh(\frac{1}{\bar{z}})
  $$
  We may choose the constants $c_1,\epsilon_1, \delta_1, c_2, \epsilon_2, \delta_2$  in the definitions of the maps $G_-(z)$ and $G_+(z)$ such that the following conditions are satisfied:
  $$
  \psi^*g=g(\frac{1}{\bar{z}})=\frac{1}{\overline{g(z)}}
  $$
  and
  $$
  \psi^*dh=dh(\frac{1}{\bar{z}})=-\overline{dh}
  $$
  Recall that $f(z)=\frac{G_-^2(z)}{z}$ and $g(z)=\frac{G_+(z)}{G_-(z)}$, and verify the following relations:\\
  
  $1.\ g(\frac{1}{\bar{z}})=\frac{G_+(\frac{1}{\bar{z}})}{G_-(\frac{1}{\bar{z}})} $ \\
  
 $
 2.\  \frac{1}{\overline{g(z)}}=\frac{\overline{G_-(z)}}{\overline{G_+(z)}}
 $ \\
 
  $
  3.\ dh(\frac{1}{\bar{z}})= f(\frac{1}{\bar{z}})g(\frac{1}{\bar{z}}) d(\frac{1}{\bar{z}}) $\\
  
  $
 4.\  -\overline{dh}=-\overline{f(z)}\ \overline{g(z)} d\bar{z}
  $\\
  
 Following a concise computation to meet the specified conditions, we obtain the following system of equations:
   \begin{equation}
      \nonumber
      \begin{cases}
           G_+(\frac{1}{\bar{z}})=\overline{G_-(z)} \\
           G_-(\frac{1}{\bar{z}})=\overline{G_+(z)}
      \end{cases}
   \end{equation}
  That results in
  
  \begin{equation}
      \nonumber \begin{cases} c_2\bar{z}+\epsilon_2+\delta_2\bar{z}^{-1}=\bar{c_1}\bar{z}+\bar{\epsilon_1}+\bar{\delta_1}\bar{z}^{-1}\\ c_1\bar{z}^{-1}+\epsilon_1+\delta_1\bar{z}=\bar{c_2}\bar{z}^{-1}+\bar{\epsilon_2}+\bar{\delta_2}\bar{z} 
     \end{cases}
   \end{equation}
   Now set 
   $$
   c_2=\bar{c_1},\ \epsilon_2=\bar{\epsilon_1},\ \delta_2=\bar{\delta_1}
   $$
  
  Compute  
  \begin{equation}
  \nonumber
  \begin{split}
  \psi^*\bold{X} &= \bold{X}\circ \psi\\
  &= Re \int_{\psi(p_0)}^{\psi(p)} \Big( \frac{1}{2}(g^{-1}\circ \psi -g\circ \psi ),\frac{i}{2}(g^{-1}\circ \psi +g\circ \psi ),1    \Big)dh\circ \psi \\
  &=Re \int_{p_0}^p \Big( \frac{1}{2}(\bar{g}-\bar{g}^{-1}),\frac{i}{2}(\bar{g}+\bar{g}^{-1}),1    \Big)(-\overline{dh})\\
  &=Re \overline{\int_{p_0}^p \Big( \frac{1}{2}(g^{-1}-g),\frac{i}{2}(g^{-1}+g),-1    \Big)dh}\\
  &=Re \int_{p_0}^p \Big( \frac{1}{2}(g^{-1}-g),\frac{i}{2}(g^{-1}+g),-1    \Big)dh \\
  &=R_3\circ \bold{X},
   \end{split}
  \end{equation}
  where $R_3$ is the reflection with respect to the plane $\{x_3=0\}$.
\end{proof}

\begin{proposition} \label{prop: length comparison in A(omega,2)}
     Assume $\bold{X} \in \mathcal{A}_V(\Omega ,2)$ is the surface $\Sigma$ described in Theorem \ref{thm: existence for L''<4L}, meeting the conditions outlined in Proposition \ref{prop: symm.surface with L''<L}, with Flux$(\Sigma)=(0,0,F_3)$. Let $C \in \mathcal{A}_V(\Omega ,2)$ be the waist of a 2-fold cover of the catenoid, symmetric with respect to reflection through the horizontal plane $\{ x_3=0\}$, and possessing the flux vector $(0,0,F_3(\Sigma))$. Here, the slab $\Omega$ is as outlined in the introduction, for sufficiently small positive number $h_+$, and $h_-=-h_+$. Then for any non-zero $h\in [h_-,h_+]$: 
     $$\ell_{\Sigma}(h)<\ell_{C}(h),$$
     where $\ell_{\Sigma}(h)=\mathcal{H}^1(\Sigma_h)$ (resp. $\ell_C(h)=\mathcal{H}^1(C_h)$), for $h\in (h_-,h_+)$, \\ $\ell_{\Sigma}(h_+)=\lim \inf_{t\nearrow h_+}  \mathcal{H}^1(\Sigma_h)$ (resp.  $\ell_{C}(h_+)$), and $\ell_{\Sigma}(h_-)=\lim \inf_{t\searrow h_-}  \mathcal{H}^1(\Sigma_h)$ (resp.  $\ell_{C}(h_-)$). Here $\Sigma_h$ and $C_h$ are level curves of $\Sigma$ and $C$, respectively.
\end{proposition}
\begin{proof}
     As shown by Bernstein-Breiner, the length function $\ell_{\Sigma}(h)$ depends smoothly on $h$, for $h\in (h_- ,h_+ )$. Defining $\ell_{\Sigma}(h_+)$ as the limit as $t$ approaches $h_+$ and $\ell_{\Sigma}(h_-)$ as the limit as $t$ approaches $h_-$, we treat $\ell_{\Sigma}(h)$ as a function defined on the interval $[h_-, h_+]$. Moreover, through the small slab $\Omega$, the length function $\ell_{\Sigma}$ is both continuous and finite over the interval $[h_-,h_+]$. Analogous properties apply to the function $\ell_C$ over the same interval. \\
     According to the results of Theorem \ref{thm: existence for L''<4L}, we have
    $$
    \ell_{\Sigma}''(h)< 4\frac{(2\pi)^2}{F_3^2}\ell_{\Sigma}(h),
    $$
    whereas, by Lemma \ref{lem: ineq. lenghth for A^*}
    $$
    \ell_C''(h)= 4\frac{(2\pi)^2}{F_3^2}\ell_C(h)
    $$
    where Flux$(\Sigma )=(0,0,F_3)$ and Flux$(C)=(0,0,F_3)$. \\
    Note that for any $h\in [h_-,h_+]$
    $$
    F_3=\Big\vert \int_{\Sigma_h}\bold{e}_3.\nu ds \Big\vert \leq \int_{\Sigma_h} ds=\ell_{\Sigma}(h)
    $$
    with equality if and only if $\Sigma_h$ is a geodesic in $\Sigma$. For the end points $h_{\pm}$ we take the limit of the corresponding integrals when $h$ approaches the limit points $h_{\pm}$, respectively.
   \medskip

    Observe that the symmetry of the surfaces $\Sigma$ and $C$ about $\{ x_3=0 \}$ indicates that $h=0$ is a critical point for both length functions. Consequently, $\ell'_{\Sigma}(0)=0$ and $\ell'_{C}(0)=0$. Furthermore, both the level curves $\Sigma_0$ and $C_0$ are geodesics on the surfaces $\Sigma$ and $C$, respectively. This implies that $F_3=F_3(\Sigma)=\ell_{\Sigma}(0)$ and $F_3=F_3(C)=\ell_{C}(0)$. Therefore, $\ell_{\Sigma}(0) = \ell_{C}(0)$.
    \medskip
    
The choice of $C$, along with the equation  $\ell_C''(h)= 4\frac{(2\pi)^2}{F_3^2}\ell_C(h)$, and the inequality $\ell_{\Sigma}''(h)< 4\frac{(2\pi)^2}{F_3^2}\ell_{\Sigma}(h)$, implies that
$$
\frac{d^2}{dh^2}(\ell_C(h)-\ell_{\Sigma}(h))> 4\frac{(2\pi)^2}{F_3^2} (\ell_C(h)-\ell_{\Sigma}(h)) 
$$
By employing the equations $\ell_{\Sigma}(0) = \ell_{C}(0)$ and $\ell'_{\Sigma}(0) = \ell'_{C}(0) =0$, and performing
a comparison of ordinary differential equations, it follows that for all $h\in [h_-,h_+],\  \ell_{\Sigma}(h) \leq \ell_C(h)$, and equality occurs only for $h= 0$. 
\end{proof}

\begin{theorem} \label{thm: main result}
    Assume $\bold{X} \in \mathcal{A}_V(\Omega ,2)$ represents the surface $\Sigma$ described in Theorem \ref{thm: existence for L''<4L}, satisfying the conditions specified in Proposition \ref{prop: symm.surface with L''<L}, with the flux vector $(0,0,F_3)$, and $C_2$ is the waist of 2-fold cover of the catenoid, symmetric with respect to the reflection through the horizontal plane, and having the flux vector $(0,0,F_3)$, it follows that
    $$
    Area(\Sigma) < Area (C_2 \cap \Omega)
    $$
\end{theorem}
\begin{proof}
    Fix the slab $\Omega$ thin relative to the vertical flux, and simplify the notation by setting $C:=C_2$. By the smoothness of the function $x_3$ and the fact $\vert \nabla_{\Sigma_0} x_3\vert =1 $, where $\Sigma_0$ is the level curve corresponding to $h=0$, we can estimate $\vert \nabla_{\Sigma} x_3\vert$ by two first terms in its Taylor series as follows
    $$
    \vert \nabla_{\Sigma} x_3\vert \sim 1+a_2h^2
    $$
where $a_2$ is a non-constant function on the level curves $\Sigma_h$. Note that the first derivative with respect to $h$ is zero at $h=0$, which gives $a_1=0$. Similarly,
$$
\vert \nabla_{C} x_3\vert \sim 1+c_2h^2
$$
where $c_2$ is a constant function on the level curve $C_{h}$.\\
    Note that
    $$
    \int_{\Sigma_h} \vert \nabla_{\Sigma}x_3 \vert =F_3= \int_{C_h} \vert \nabla_{C}x_3 \vert 
    $$
    on all level curves $[h_-,h_+]$, and recall that $\ell_{\Sigma}(h)(=\mathcal{H}^1 (\Sigma_h)) < \ell_{C} (h)(=\mathcal{H}^1 (C_h)) $, for all non-zero $h\in [h_-,h_+]$. We compute
    $$
    \int_{\Sigma_h} \vert \nabla_{\Sigma}x_3 \vert \sim \int_{\Sigma_h} 1+a_2h^2 = \ell_{\Sigma}(h) +h^2\int_{\Sigma_h}a_2
    $$
    $$
   \int_{C_h} \vert \nabla_{C}x_3 \vert  \sim \int_{C_h} 1+c_2h^2 = \ell_C(h) +h^2\int_{C_h}c_2
    $$
    Note that the length functions $\ell_{\Sigma}(h)$ and $\ell_C(h)$ are smooth functions on the domain $[h_-,h_+]$, we can estimate by first terms in the Taylor series
$$
\ell_{\Sigma}(h)\sim \ell_{\Sigma}(0) +\sigma h^2
$$
and,
$$
\ell_C(h)\sim \ell_C(0) + ch^2
$$ 
Recall that the first derivative of the length functions with respect to $h$ is zero at $h=0$, which give the first non-constant terms in $h$ to be zero.
 Then,
 $$
  \int_{\Sigma_h} \vert \nabla_{\Sigma}x_3 \vert \sim \ell_{\Sigma}(0) + \sigma h^2+h^2\int_{\Sigma_h}a_2
 $$
 $$
   \int_{C_h} \vert \nabla_{C}x_3 \vert  \sim \ell_C(0) + ch^2 +h^2\int_{C_h}c_2
    $$
  By the fact that $F_3=\ell_{\Sigma}(0)=\ell_C(0)$, we conclude
  $$
  \sigma = - \int_{\Sigma_h}a_2,\ \ \text{and},\ \ c= -\int_{C_h}c_2,
  $$
where $\sigma < c$ by the proposition 3.3 for $h\neq 0$.\\
 Let $h_0 \in (h_-,h_+)$ and define the function $A_{\Sigma ,h_0}$ on $[h_0,h_+)$ by
    $$
    A_{\Sigma ,h_0}(h) = \mathcal{H}^2(\Sigma \cap \{ h_0 \leq x_3 <h\})
    $$
    and the function $A_{C,h_0}$ similarly.\\
    The co-area formula implies that 
    $$
    \frac{d}{dh}A_{\Sigma ,h_0} (h) = \int_{\{ x_3=h\} \cap \Sigma} \frac{1}{\vert \nabla_{\Sigma}x_3 \vert}
    $$
We write
$$
 \int_{\Sigma_h} \frac{1}{\vert \nabla_{\Sigma}x_3 \vert} \sim \int_{\Sigma_h} \frac{1}{1+a_2h^2} \sim \int_{\Sigma_h} 1-a_2h^2 
$$
 and,  
   $$
 \int_{C_h} \frac{1}{\vert \nabla_{C}x_3 \vert} \sim \int_{C_h} \frac{1}{1+c_2h^2} \sim \int_{C_h} 1-c_2h^2 
$$
 Then,
$$
 \int_{\Sigma_h} \frac{1}{\vert \nabla_{\Sigma}x_3 \vert} -  \int_{C_h} \frac{1}{\vert \nabla_{C}x_3 \vert} \sim \int_{\Sigma_h} (1-a_2h^2) - \int_{C_h} (1-c_2h^2) = (\sigma -c)h^2+(\sigma -c) <0
$$ 
where $h\neq 0$.\\
  Hence,
    $$
    \frac{d}{dt} A_{\Sigma ,h_0}(t) - \frac{d}{dt} A_{C ,h_0}(t) <0
    $$
    for all choices of $0\neq h_0$ and $h$, other than for $h_0=0$ when $h$ is tending to $0$. However, in the case of $h_0=0$ and $h\rightarrow 0$, it follows that
    $$
   \lim_{h\rightarrow 0} \frac{d}{dh} A_{\Sigma ,0}(h) \leq \lim_{h\rightarrow 0} \frac{d}{dh} A_{C ,0}(h).
    $$
    Letting $h_0 \rightarrow h_-$ and integrating over $h$ implies that 
   
    $$
   Area(\Sigma) < Area (C)= Area (C_2).
    $$ 
\end{proof}

\section{A note on the space $\mathcal{A}(\Omega,0)$}
 We conclude this note by exploring the set $\mathcal{A}_V(\Omega, 0)$. Notice that the surfaces within this set do not have a corresponding catenoid cover for comparison.
 Setting $k=0$, we observe that among all surfaces in $\mathcal{A}_V(\Omega, 0)$, those annuli featuring level curves resembling the figure eight exhibit the least self-intersections. It is expected that these featured surfaces within the set $\mathcal{A}_V(\Omega, 0)$ achieve the minimum area. 

\begin{theorem} \label{thm: annulus in A(omega, 0)}
   There exists an annulus $\bold{X}\in \mathcal{A}(\Omega ,0)$, denoted as $\Sigma$, featuring 
   the level curves $\Sigma_h$ of figure eight, which satisfy the following convexity on their length function
   $$
   2\ell_{\Sigma}(h)<\ell_{\Sigma}''(h)<4\ell_{\Sigma}(h)
   $$
   Additionally, this annulus $\bold{X} \in \mathcal{A}(\Omega ,0)$ is a part of a complete minimal surface in $\mathbb{R}^3$ with a total curvature equal to $-8\pi$, featuring two ends, each with a multiplicity of two.
\end{theorem}
\begin{proof}
    We fix $\Omega$ as an open slab, as defined earlier in this note, and adhere to the notations and terms introduced previously.
    Let $\Sigma$ be a surface parameterized by $\bold{X}\in \mathcal{A}_V(\Omega ,0)$, for a conformal minimal mapping $\bold{X} : A_{R_0,R_1} \rightarrow \mathbb{R}^3$, serving as the corresponding Weierstrass representation. Here, the domain $A_{R_0,R_1}$ is sufficiently small and bounded. Let the flux vector for such a minimal annulus be denoted as $(0,0,F_3)$. Note that, according to Lemma \ref{lem: L''(t)>2L(t)}, we observe $L''(t) > 2L(t)$. To identify surfaces that exhibit the minimum growth rate of the length function, we explore potential surfaces that meet the criteria: 
\begin{equation} \label{ineq: L for wind. 0}
2L(t) < L''(t) < 4L(t),\ \forall t
\end{equation}
Alternatively, the associated conformal parameterization satisfies the following condition: 
\begin{equation}
    \nonumber
    2(\frac{2\pi}{F_3(\Sigma)})^2 \ell_{\Sigma} (h)< \ell_{\Sigma}'' (h) < 4(\frac{2\pi}{F_3(\Sigma)})^2 \ell_{\Sigma}(h),
\end{equation}
where $F_3(\Sigma)$ is the third component of the flux vector, and $h\in (h_{-},h{+})$. By a homothetical rescale of $\Sigma$, we can make $F_3=2\pi$, hence we make the assumption that:
$$
2 \ell_{\Sigma} (h)< \ell_{\Sigma}'' (h) < 4 \ell_{\Sigma}(h),
$$
Note that this is equivalent to 
\begin{equation} \label{ineq: F for win. 0}
   2\int_{0}^{2\pi} \vert F_{\pm} \vert d\theta < \int_{0}^{2\pi} r^2 \Delta \vert F_{\pm} \vert d\theta < 4\int_{0}^{2\pi} \vert F_{\pm} \vert d\theta ,  
\end{equation}

Recall that $F_{\pm}(z)=G_{\pm}^2(z)$, and the Gauss map $g(z)$ can be expressed as $g(z) = \frac{G_+(z)}{G_-(z)}$.
 To identify the best possible $F_{\pm}(z)$ that satisfies the inequalities (\ref{ineq: F for win. 0}), we need to examine the corresponding $G_{\pm}(z)$. Let $G(z)$ be analytic in the annulus $A_{R_0,R_1}$ and have there the Laurent expansion $G(z)=\underset{n=-\infty}{\overset{\infty}{\sum}}a_nz^n$. By comparing the two inequalities (\ref{ineq: int. norm.sqr.G}) and (\ref{ineq: F for win. 0}), we verify that the optimal candidates for the analytic functions $G_{\pm}(z)$ are polynomials of the form $G(z)= a_{-1}z^{-1} + a_0 + a_1z$, where $a_{-1}\neq 0, a_0\neq 0, a_1\neq 0$. 
 It should be noted that if we set any $a_i$ equal to zero in the equation for $G(z)$, for $i=-1,0,1$, it leads to an equation for the Gauss map associated with a surface that does not belong to $\mathcal{A}_V(\Omega ,0)$. Note that, we need to choose the domain $A_{R_0,R_1}$ and the constants $a_{-1},a_0,a_1$ so that the Gauss map $g(z)=\frac{G_+(z)}{G_-(z)}$ has no pole or zero in the domain $A_{R_0,R_1}$.
\medskip

 Consider $F(z)=G(z)^2 $, where $G(z)= a_{-1}z^{-1} + a_0 + a_1z$. Hence, $F(z)$ has neither poles nor zeros in the domain.  It is essential to ensure that $F(z)$ satisfies the condition associated with the period problem, making the resulting surface well-defined. This condition is expressed as $\int_{0}^{2\pi} F(re^{i\theta})d\theta =0$, which also ensures that the flux is vertical. By computing $F(z)=a_{-1}^{2}z^{-2}+2a_{-1}a_0z^{-1}+ a_0^2 + 2a_{-1}a_1 + 2a_0a_1z+a_1^2z^2$, we can apply the well-defined assumption to deduce that the constant term must vanish, leading to the equation $a_0^2 + 2a_{-1}a_1 =0$. 
By choosing the slab $\Omega$ thin enough, and so the domain $A_{R_0,R_1}$ small, we can ensure that the surface $\Sigma$ keeps a constant winding number in $\Omega$. Specifically, we may choose a thin slab, still denoted as $\Omega$, delimited by the original slab $\Omega$, and symmetric through reflection with respect to the horizontal plane. Set
$$
\Sigma := \Sigma \cap \Omega
$$
Here $\Omega$ is the new sufficiently thin slab, then
$$
\bold{X}^{-1}(\Sigma)=A ,
$$
for some topological annulus. By the Riemann mapping theorem there exists a conformal diffeomorphism $T$ from some planar annulus $A_{R_2,R_3}$ onto $A$, for some positive values $R_2 < R_3$, and $A_{R_2,R_3} \subset A_{R_0,R_1}$.\\
Consider the map $\bold{X}\circ T: A_{R_2,R_3}\rightarrow \mathcal{A}_V(\Omega ,0)$ and note that for a sufficiently thin slab $\Omega$, and subsequently, a thin annulus $A$, the map $T$ is close to the identity. Hence, $\bold{X}\circ T$ is a parameterization that approximates the same Weierstrass data.\\
In particular, computing the winding number, $k$, of the circles foliating the surface $\Sigma$ by:
$$
\int_0^{2\pi} \frac{\partial}{\partial \theta} arg F(re^{i\theta}) d\theta =2\pi k
$$
For the function $F(z)$, it can be verified that this integral is equal to zero. Hence, $\bold{X}\circ T(A_{R_2,R_3})$, and the surfaces constructed by these data are in $\mathcal{A}_V(\Omega ,0)$. 
\medskip

To construct the surface using the Weierstrass representation mapping method, consider the polynomials, $G_-(z) = a_{-1}z^{-1} + a_0 + a_1z$ and $G_+(z) = b_{-1}z^{-1} + b_0 + b_1z$ so that $a_i\neq 0$ and $b_i\neq 0$ for $i=-1,0,1$. Moreover, choose $a_i$ and $b_i$, for $i=-1,0,1$, so that $a_0^2 + 2a_{-1}a_1 =0$ and $b_0^2 + 2b_{-1}b_1 =0$, and ensure that $G_+(z)$ and $G_-(z)$ have no pole or zero in the domain. We compute $g(z)=\frac{b_{-1}z^{-1}+b_0+b_1z}{a_{-1}z^{-1}+a_0+a_1z}$ and $f(z)=z^{-1}(a_{-1}z^{-1}+a_0+a_1z)^2$. 
Then, set $F_{\pm} = G_{\pm}^2 (z)$.\\
Considering the equations (\ref{L"}) and (\ref{L}), we obtain the expressions:
\begin{equation}
    \nonumber
    \begin{split}
    & \frac{\psi_3}{g} = F_-(z) = a_{-1}^{2}z^{-2}+2a_{-1}a_0z^{-1}+2a_0a_1z+a_1^2z^2,  \\  
   & \psi_3 g = F_+(z) = b_{-1}^{2}z^{-2}+2b_{-1}b_0z^{-1}+2b_0b_1z+b_1^2z^2,
    \end{split}
\end{equation} 
where $\psi_3 =z\phi_3$.
And, the associated Weierstrass data are as follows:
\begin{equation}
    \nonumber
        \phi_1 = \frac{1}{2} f(1-g^2),\ \ \
        \phi_2 = \frac{i}{2} (1+g^2),\ \ \
        \phi_3 = fg    
\end{equation}

It follows that
\begin{equation}
    \nonumber
    \begin{split}
       & \phi_1 = \frac{1}{2} z^{-1} \big( (a_{-1}z^{-1}+a_0+a_1z)^2 - (b_{-1}z^{-1}+b_0+b_1z)^2 \big)   \\
       & \phi_2 = \frac{i}{2} z^{-1} \big( (a_{-1}z^{-1}+a_0+a_1z)^2 + (b_{-1}z^{-1}+b_0+b_1z)^2 \big)   \\
      &  \phi_3 = z^{-1}  (a_{-1}z^{-1}+a_0+a_1z)  (b_{-1}z^{-1}+b_0+b_1z) 
    \end{split}
\end{equation}
Simplifying,
\begin{equation}
    \nonumber
    \begin{split}
       & \phi_1 = \frac{a_{-1}^2-b_{-1}^2+2(a_{-1}a_0-b_{-1}b_0)z+2(a_0a_1-b_0b_1) z^3 +(a_1^2-b_1^2)z^4}{2z^3}      \\
       & \phi_2 = \frac{i(a_{-1}^2+b_{-1}^2+2(a_{-1}a_0+b_{-1}b_0)z+2(a_0a_1+b_0b_1) z^3 +(a_1^2+b_1^2)z^4)}{2z^3}   \\
      &  \phi_3 = \frac{a_{-1}b_{-1}+(a_{-1}b_0+a_0b_{-1})z+(a_{-1}b_1+a_0b_0+a_1b_{-1})z^2+(a_0b_1+a_1b_0)z^3+a_1b_1z^4}{z^3}  
    \end{split}
\end{equation}

Here, $a_0^2 + 2a_{-1}a_1 =0$ and $b_0^2 + 2b_{-1}b_1 =0$, and $a_i\neq 0$,  $b_i\neq 0$ for $i=-1,0,1$.
\medskip

Consider the mapping $\bold{X}$ on the domain $A_{R_2,R_3}$, where the domain is selected based on the criteria described above, so that $\bold{X}$ represents an annulus with a winding number zero through $A_{R_2,R_3}$. We can compute the zeros of the complexified height differential $\phi_3$ by solving equations  $a_{-1}z^{-1}+a_0+a_1z=0$ and $b_{-1}z^{-1}+b_0+b_1z=0$. Let $z_1,z_2$ and $z_3,z_4$ be the corresponding zeros, respectively. We can verify that $z_j \in \mathbb{C}$ for $j=1,2,3,4$ lie on four distinct circles for some appropriate choices of $a_i$s and $b_i$s, where $i=-1,0,1$. We choose $A_{R_2,R_3}$ so that the points $z_1$ and $z_3$ are inside the disk $D_{R_2}(0)$ and the points $z_2$ and $z_4$ are outside the disk $D_{R_3}(0)$. Hence, by the Cauchy-Residue theorem the winding number of $g(z)$ is zero on $A_{R_2,R_3}$, and this ensures that the surface $\bold{X}$ has a constant winding number on the domain $A_{R_2,R_3}$. Furthermore, we can arrange the surface $\bold{X}$ to be symmetric with respect to the horizontal plane $\{ x_3=0\}$ through reflection, see Step One below, which ensures that the flux is vertical.
\medskip

The mapping $\bold{X}(z)=\bold{X}(z_0)+\text{Re} \int_{z_0}^z(\phi_1,\phi_2,\phi_3)$ defined on the annulus $A_{R_2,R_3}$ can be extended to the complex plane $\mathbb{C}$ minus a point by letting $R_2 \rightarrow 0$ and $R_3 \rightarrow \infty$ on $A_{R_2,R_3}$. However, the winding number may change as the domain expands; in fact, it changes at the points where the complexified height differential $\phi_3$ vanishes. Additionally, expanding the domain can cause the Gauss map to develop poles or zeros. In fact, the corresponding expanded minimal surface is conformal to the sphere minus two points. Those two points correspond to the ends at $\infty$ and $0$, each with a multiplicity of two. Therefore, by the Jorge-Meeks formula, the surface is of total curvature $-8\pi$. This is a complete minimal surface. Furthermore, integration of the Weierstrass data once around $0$ adds the period $(0,0,2\pi i)$ to $\int (\phi_1,\phi_2,\phi_3)$. Hence, the surface described with the map $\bold{X}(z)$ is defined on $\mathbb{C}- \{0 \}$.

The interesting fact about this surface is that, up to rotation and rigid motion, a part of the surface intersecting the slab $\Omega$ belongs to $\mathcal{A}_V(\Omega ,0)$. More interesting, the level curves of this specific section of the surface take the form of a figure eight. 
\end{proof}

\begin{corollary} 
    If we consider $\Sigma$ and $\bold{X} \in \mathcal{A}_V(\Omega ,0)$ as a surface constructed according to Theorem \ref{thm: annulus in A(omega, 0)}, then the expression for $\ell_{\Sigma}''(h)$ is given by the following equation: 
    \begin{equation} \label{coro: 2nd ord.eqn. involving F3/2}
    \ell_{\Sigma}''(h) = 4(\frac{2\pi}{F_3(\Sigma)})^2 \ell_{\Sigma}(h) - \frac{1}{\pi} (\vert c_-\vert^2 + \vert c_+ \vert^2),
    \end{equation}
    where $c_-:=\int_0^{2\pi} G_-(z)d\theta = 2\pi a_0$ and $c_+ :=\int_0^{2\pi} G_+(z)d\theta = 2\pi b_0$.
\end{corollary}
\begin{proof}
    It can be readily verified that the functions $G_-$ and $G_+$, as defined in Theorem $\ref{thm: annulus in A(omega, 0)}$, meet the following conditions:
\begin{equation}
    \nonumber
    \begin{split}
    & \int_0^{2\pi} r^2 \Delta \vert G_-\vert^2 d\theta = 4\int_0^{2\pi} \vert G_-\vert^2 d\theta - \frac{2}{\pi} \vert c_- \vert^2 \\
    & \int_0^{2\pi} r^2 \Delta \vert G_+\vert^2 d\theta = 4\int_0^{2\pi} \vert G_+\vert^2 d\theta - \frac{2}{\pi} \vert c_+ \vert^2
    \end{split}
\end{equation}
By comparing these equations to equations (\ref{L"}) and (\ref{L}), we observe that: 
\begin{equation} \label{conv. L for win. 0}
   L''(t) = 4L(t) - \frac{1}{\pi} (\vert c_- \vert^2 + \vert c_+ \vert^2) 
\end{equation}
The assertion presents an alternative expression for this equation.
\end{proof}

As mentioned before, there is no zero-cover catenoid available for comparison with the surfaces in the set $\mathcal{A}_V(\Omega,0)$. However, in the following analysis, we compare the surfaces in $\mathcal{A}_V(\Omega,0)$ with the waists of catenoid with a winding number one. Specifically, in the following theorem, we address Problem 2. from \cite{choe2016area}, which concerns the comparison between the area of catenoid waists and the area of the annuli having figure-eight level curves. In the following theorem \ref{thm: area(k=0) > C}, the horizontal slab $\Omega$ may not necessarily be symmetric with respect to reflection through the horizontal plane $\{x_3=0\}$, namely $h_-$ is not necessarily equal to $-h_+$.

\begin{theorem} \label{thm: area(k=0) > C}
   Let $P_-=\{ x_3=h_-\}$ and $P_+=\{ x_3=h_+\}$ be distinct parallel planes in $\mathbb{R}^3$ with $h_-<h_+$ and let $\Omega$ be the open slab between them.
   Suppose that $\Sigma$ is an annulus parameterized with $\bold{X}\in \mathcal{A}_V(\Omega ,0)$  and assume $Flux(\Sigma )=(0,0,F_3)$. Let $P_0=\{ x_3=h_0\} \subset \Omega$ denote the plane that satisfies:
   $$
   \ell_{\Sigma}(h_0)=\underset{h\in (h_-,h_+)}{\inf} \ell_{\Sigma}(h)
   $$
   Let $C \in \mathcal{A}_V(\Omega ,1)$ be the waist of the complete catenoid which is symmetric with respect to reflection through the plane $P_0$, and also has $Flux(C)=(0,0,F_3)$. Then
   $$
   Area (\Sigma) > Area (C)
   $$
   in particular,
   $$
   Area (\Sigma ) > Area (C_{\Omega})
   $$
   where $C_{\Omega}$ is the maximally symmetric marginally stable waist of the catenoid spanning $\Omega$. 
\end{theorem}
\begin{proof}
    By Lemma \ref{lem: L''(t)>2L(t)}, for any annulus $\Sigma \in \mathcal{A}_V(\Omega ,0)$ the length function of the level curves satisfies $L''(t)> 2L(t)$, or equivalently $\ell''_{\Sigma}(h)>2 \frac{(2\pi)^2}{F_3^2}\ell_{\Sigma}(h)$. On the other hand, for the catenoid $C\in \mathcal{A}_V(\Omega ,1)$ the length function satisfies $\ell''_{C}(h)=\frac{(2\pi)^2}{F_3^2}\ell_{C}(h)$.  First suppose that $h_0\in (h_-,h_+)$, using the fact that $C$ is symmetric with respect to reflection through the plane $P_0$, an argument similar to the proof of the proposition \ref{prop: length comparison in A(omega,2)} shows that: 
   $$
   \ell_{\Sigma}(h) > \ell_C(h),\ \ \ \forall h(\neq h_0) \in (h_-,h_+)
   $$
   Now suppose that $h_0$ is an endpoint, and assume $h_0=h_-$. For $\epsilon >0$ small, set $\ell_{C,\epsilon}(h)=\ell_{C}(h-\epsilon)$ and restrict to $(h_-+\epsilon ,h_+)$. Then, $\ell_{C,\epsilon}''=\frac{(2\pi)^2}{F_3^2}\ell_{C,\epsilon}$, $\ell_{C,\epsilon}'(h_-+\epsilon)=0$ and $\ell_{\Sigma}(h_-+\epsilon)>\ell_{\Sigma}(h_-)\geq \ell_{C,\epsilon}(h_-+\epsilon)$. Note that $\ell_{\Sigma}$ is convex and has its minimum at $h_-$, so $\ell_{\Sigma}'(h_-+\epsilon)>0$. Hence, by an ODE comparison, $\ell_{\Sigma}(h)>\ell_{C,\epsilon}(h)$ for $h\in [h_-+\epsilon ,h_+)$. Letting $\epsilon \rightarrow 0$ implies $\ell_{\Sigma}(h)> \ell_{C}(h)$ for $h\in (h_-,h_+)$.  An identical argument applies when $h_0=h_+$.
    \medskip
   
   Let $h_*\in (h_-,h_+)$ and define the function $A_{\Sigma ,h_*}$ on $[h_*,h_+)$ by
   $$
   A_{\Sigma ,h_*}(h)=\mathcal{H}^2(\Sigma \cap \{ h_* \leq x_3 <h\})
   $$
   and the function $A_{C,h_*}$ similarly.\\
   The co-area formula implies that 
   $$
   \frac{d}{dh}A_{\Sigma ,h_*}(h)= \int_{\{ x_3=h\}} \frac{1}{\vert \nabla_{\Sigma}x_3\vert}
   $$
By the Cauchy- Schwarz inequality:
$$
\frac{d}{dh}A_{\Sigma ,h_*}(h) \geq \frac{\ell_{\Sigma}^2(h)}{\int_{\Sigma_h}\vert \nabla_{\Sigma}x_3\vert}= \frac{\ell_{\Sigma}^2(h)}{F_3}
$$
One has equality if and only if $\frac{1}{\vert \nabla_{\Sigma}x_3\vert}$ and $\vert \nabla_{\Sigma}x_3\vert$ are linearly dependent, so are both constant. This is the case on the catenoid $C$, so using the above estimate for length:
$$
\frac{d}{dh}A_{\Sigma ,h_*}(h) >  \frac{\ell_{C}^2(h)}{F_3} = \frac{d}{dh}A_{C ,h_*}(h)
$$
By integrating, we find that $\frac{d}{dh}A_{\Sigma ,h_*}(h) > \frac{d}{dh}A_{C ,h_*}(h)$ for $h\in [h_*,h_+)$. Letting $h_* \rightarrow h_-$  proves the theorem for the catenoid waist $C$ described above. Additionally, Theorem 1.1 in \cite{bernstein2014variational} establishes the result for the maximally symmetric marginally stable waist of the catenoid spanning $\Omega$, denoted as $C_{\Omega}$.
\end{proof}

\begin{rem}
    \rm  By Lemma \ref{lem: L''(t)>2L(t)}, a similar argument shows that Theorem \ref{thm: area(k=0) > C}, in fact, holds for any $\bold{X} \in \mathcal{A}_V(\Omega ,2k)$, $k\geq 0$.
\end{rem}

We then conclude this section by comparing the area of the waists of two-fold covers of the catenoid with the area of an example from Theorem  \ref{thm: annulus in A(omega, 0)}.
Fix the slab $\Omega$ as described in the introduction. We continue to assume that the positive number, $h_+=-h_-$, associated with the slab $\Omega$, is suitably small. Employing reasoning analogous to that used for the surfaces in $\mathcal{A}_V(\Omega ,2)$ constructed in Section 3, we identify the surfaces in $\mathcal{A}_V(\Omega ,0)$ that achieve an area smaller than the double cover of the waist of a catenoid spanning $\Omega$ and having the same $F_3$. Here, $F_3$ still denotes the third component of the surface's flux vector. This can be accomplished by taking two major steps: 

\begin{theorem}
    Assume $\bold{X} \in \mathcal{A}_V(\Omega ,0)$ is an annulus constructed in Theorem \ref{thm: annulus in A(omega, 0)}, which is symmetric with respect to reflection through the horizontal plane $\{ x_3=0\}$, with $Flux(\Sigma)=(0,0,F_3)$. Suppose that $C_2$ is a 2-fold cover of the waist the catenoid, symmetric with respect to the reflection through the horizontal plane $\{ x_3=0\}$, and $C_2$ having the flux vector $(0,0,F_3)$, it follows that
    $$
    Area(\Sigma) < Area (C_2 \cap \Omega).
    $$
\end{theorem}

The proof follows by taking the following steps:

\textit{Step One}:
\label{stepone: symmetry in A(omega,0)}
   For the annuli in $\mathcal{A}_V(\Omega ,0)$ as constructed in Theorem \ref{thm: annulus in A(omega, 0)}, we may choose the constants $a_{-1},a_0,a_1,b_{-1},b_0$ and $b_1$ so that the annulus exhibits symmetry with respect to reflection through the horizontal plane $\{ x_3 =0 \} \subset \Omega$. 

\begin{proof} (\textit{Step One})
     A reasoning parallel to the proof of Proposition \ref{prop: symm.surface with L''<L} can be employed. We deduce that by setting $a_0=\bar{b}_0$, $a_{-1} =\bar{b}_1$, and $a_1 = \bar{b}_{-1}$, there exists a slab with the characteristics described, thereby establishing the desired result.
\end{proof}

Then the same reasoning as in Proposition \ref{prop: length comparison in A(omega,2)} applies to prove the following claim.\\
\textit{Step Two}:
     Assume $\Sigma \in \mathcal{A}_V(\Omega ,0)$ is an annulus constructed in Theorem \ref{thm: annulus in A(omega, 0)}, meeting the conditions outlined in Step One, with $Flux(\Sigma)=(0,0,F_3)$. Let $C \in \mathcal{A}_V(\Omega ,2)$ be the waist of a 2-fold cover of the complete catenoid which is symmetric with respect to reflection through the horizontal plane $\{ x_3=0\}$, and has the flux vector $(0,0,F_3(\Sigma))$. Then for any non-zero $h\in [h_-,h_+]$: 
     $$\ell_{\Sigma}(h)<\ell_{C}(h),$$
     where $\ell_{\Sigma}(h)=\mathcal{H}^1(\Sigma_h)$ (resp. $\ell_C(h)=\mathcal{H}^1(C_h)$), for $h\in (h_-,h_+)$, \\ $\ell_{\Sigma}(h_+)=\lim \inf_{t\nearrow h_+}  \mathcal{H}^1(\Sigma_h)$ (resp.  $\ell_{C}(h_+)$), and $\ell_{\Sigma}(h_-)=\lim \inf_{t\searrow h_-}  \mathcal{H}^1(\Sigma_h)$ (resp.  $\ell_{C}(h_-)$). Here, $\Sigma_h$ and $C_h$ are level curves of $\Sigma$ and $C$, respectively.
\medskip

Then an argument similar to the proof of Theorem \ref{thm: area(k=0) > C} applies to complete the proof.
\medskip

Finally, a related question is:\\

\textit{Can one characterize the area minimizers in $\mathcal{A}_V(\Omega ,0)$? \\
Are these minimizers related to the surfaces constructed in Theorem \ref{thm: annulus in A(omega, 0)} ?}

\subsection*{Acknowledgment}
The author is grateful to her advisor, Jacob Bernstein, for suggesting the problem, helpful insights, guidance, and encouragement during the work.

\bibliographystyle{amsalpha}

\begin{thebibliography}{7}

\bibitem{bernstein2014variational}
Bernstein, J. and Breiner, C.
A variational characterization of the catenoid.
\textit{Calculus of Variations and Partial Differential Equations}, 49(1-2):215--232, 2014.

\bibitem{choe2016area}
Choe, J. and Daniel, B.
On the area of minimal surfaces in a slab.
\textit{International Mathematics Research Notices}, 2016(23):7201--7211, 2016.

\bibitem{hoffman1997complete}
Hoffman, D. and Karcher, H.
Complete embedded minimal surfaces of finite total curvature.
In \textit{Geometry, V}, volume 90 of \textit{Encyclopaedia of Mathematical Sciences}, pages 5--93. Springer, Berlin, 1997.

\bibitem{lopez1991embedded}
L{\'o}pez, F. J. and Ros, A.
On embedded complete minimal surfaces of genus zero.
\textit{Journal of Differential Geometry}, 33(1):293--300, 1991.

\bibitem{osserman1975doubly}
Osserman, R. and Schiffer, M.
Doubly-connected minimal surfaces.
\textit{Archive for Rational Mechanics and Analysis}, 58(4):285--307, 1975.

\bibitem{pyo2010minimal}
Pyo, J.
Minimal annuli with constant contact angle along the planar boundaries.
\textit{Geometriae Dedicata}, 146:159--164, 2010.

\bibitem{schoen1983uniqueness}
Schoen, R. M.
Uniqueness, symmetry, and embeddedness of minimal surfaces.
\textit{Journal of Differential Geometry}, 18(4):791--809, 1984.

\end{thebibliography}

\providecommand{\bysame}{\leavevmode\hbox to3em{\hrulefill}\thinspace}
\providecommand{\MR}{\relax\ifhmode\unskip\space\fi MR }
\providecommand{\MRhref}[2]{%
  \href{http://www.ams.org/mathscinet-getitem?mr=#1}{#2}
}
\providecommand{\href}[2]{#2}

\end{document}